\documentclass[12pt]{article}

\usepackage{amssymb, amsthm, mathrsfs, amsmath}

\newtheorem{theorem}{Theorem}
\newtheorem{lemma}{Lemma}

\newtheorem{corollary}{Corollary}

\def\R{{\mathbb{R}}}
\def\E{{\mathbb{E}}}
\def\ci{{\mathbf{c}}}
\def\Prob{{\mathbf{Prob}}}

\def\P{{\cal{P}}}
\def\CP{{\cal{CP}}}

\def\N{{\rm{WNC}}}
\def\C{{\rm{WCC}}}
\def\B{{\rm{Ball}}}
\def\Vol{{\mbox{The volume of }}}
\def\Vs{{\rm{Vs}}}

\def\NA{{\rm{SNC}}}
\def\CA{{\rm{SCC}}}
\def\LNA{{\rm{SNLP}}}
\def\LCA{{\rm{SCLP}}}

\begin{document}

\title{Stochastic Vs Worst-case Condition Numbers}

\author{Dennis Cheung and Lisa H.Y. Zhou\\
United International College\\
Tang Jia Wan\\
Zhuhai, Guandong Province\\
P.R. of CHINA\\
e-mail:  \texttt{dennisc@uic.edu.hk and lisazhou@uic.edu.hk}}
%\and
%Felipe Cucker\\
%Department of Mathematics\\
%City University of Hong Kong\\
%83 Tat Chee Avenue, Kowloon\\
%HONG KONG\\
%e-mail: \texttt{macucker@math.cityu.edu.hk}}

\maketitle
\begin{abstract}
We compare Stochastic and Worst-case condition numbers and loss of
precision for general computational problems. We show an upper bound
for the ratio of Worst-case condition number to the Stochastic
condition number of order $O(\sqrt m)$. We show an upper bound for
the difference between the Worst-case loss of precision and the
Stochastic loss of precision of order $O(\ln m)$. The results hold
if the perturbations are measured norm-wise or componentwise.
\end{abstract}

\section{Introduction}
Let $x\in\R^m$ and $f(x)\in\R^m$ be the input and output of a
computational problem respectively. We assume $f(x)$ is
differentiable. Condition numbers are real numbers measuring the
sensitivity of the output $f(x)$ to the input $x$ of the problem.
But there are many different versions of condition numbers. Below
are the definitions of {\bf Worst-case Norm-wise, Worst-case
Component-wise, Stochastic Norm-wise and Stochastic Component-wise}
condition numbers. For any set $\cal S$, we write $x\sim \cal S$ if
$x$ is a random variable (or vector) uniformly distributed in $\cal
S$. And we denote by $\underset {x\sim\cal S}{\E}f(x)$ the expected
value of $f(x)$ when $x\sim\cal S$.
\begin{eqnarray*}
\N(x)&=&\lim_{\delta\rightarrow 0}\sup_{x'\in\P(x, \delta)}
\frac{\|f(x')-f(x)\|}{\delta\|f(x)\|},\\
\C_j(x)&=&\lim_{\delta\rightarrow 0}\sup_{x'\in\CP(x, \delta)}
\frac{|f_j(x')-f_j(x)|}{\delta|f_j(x)|}, \\
\NA(x)&=&\lim_{\delta\rightarrow 0}\,\underset {x'\sim\P(x,
\delta)}{\E} \frac{\|f(x')-f(x)\|}{\delta\|f(x)\|}\mbox{ and }\\
\CA_j(x)&=&\lim_{\delta\rightarrow 0}\underset{x'\sim\CP(x,
\delta)}{\E}
\frac{|f_j(x')-f_j(x)|}{\delta|f_j(x)|}\mbox{,  where }\\
\P(x,\delta)&=&\{x'\in\R^m:\, \|x'-x\|\leq \delta\|x\|\}\mbox{ and }\\
\CP(x,\delta)&=&\{x'\in\R^m:\, |x'_i-x_i|\leq \delta|x_i|\mbox{ for
$i=1, ...,m$\}.}
\end{eqnarray*}

There are two reasons why researchers study condition numbers.
First, when we input a real number $x_1$ in a computer, the computer
can never store $x_1$ with $100\%$ accuracy. Instead, an approximate
value $x'_1$ will be stored. How accurate we can store a number
depends on the data type chosen for storing the number $x_1$.
Suppose we store this number with the data type {\bf double}. The
relative error
$$
\frac{|x'_1-x_1|} {|x_1|}\leq \frac{1}{2^{52}}\mbox{, which means }
x'\in\CP\left(x, \frac{1}{2^{52}}\right).
$$
So, if $\C_1(x)=4$, we can {\bf ensure}
$$
\frac{|f_1(x')-f_1(x)|}{|f_1(x)|}\leq
4\times\frac{1}{2^{52}}=\frac{1}{2^{50}}.
$$
Similarly, if $\CA_1(x)=4$, we can {\bf expect}
$$
\frac{|f_1(x')-f_1(x)|}{|f_1(x)|}=
4\times\frac{1}{2^{52}}=\frac{1}{2^{50}}.
$$

The second reason of studying condition numbers is related to the
stability of algorithms. Even if we can stored the input $x$
accurately, we still cannot find the output $f(x)$ with $100\%$
accuracy. It is because errors appear and accumulate after every
operation (addition, subtraction and etc.) done in a computer. How
accurate we can compute $f(x)$ depends on the algorithm applied. We
say that an algorithm is {\bf backward stable} if the computed
output $f'$ satisfy the following.
$$
f'=f(x')\mbox{ for some $x$, s.t. } \|x'-x\|\leq \|x\|\,
O(\varepsilon_{\rm machine}),
$$
where $\varepsilon_{\rm machine}$ is the upper bound for the
relative error occuring after one operation done. As a result,
applying a backward stable algorithm, one can {\bf ensure} the
computed solution $f'$ satisfy the following.
$$
\frac{\|f'-f(x)\|}{\|f(x)\|}\leq \N(x) O(\varepsilon_{\rm machine}).
$$
Or one can {\bf expect}
$$
\frac{\|f'-f(x)\|}{\|f(x)\|}= \NA(x) O(\varepsilon_{\rm machine}).
$$

Unless specified, $\log(x)$ refer to the logarithm with base $2$.
$(\log |x|-\log|x'-x|)$ is called the precision of $x'$. Roughly
speaking, it is the number of trustable (or accurate) bits.
$\log\N(x)$ is called the {\bf Worst-case Norm-wise Loss of
Precision} since $\log\N(x)$
\begin{eqnarray*}
 &=& \lim_{\delta\rightarrow 0}\sup_{x'\in\P(x, \delta)}
(\log|x|-\log|x'-x|)-(\log|f(x)|-\log|f(x')-f(x)|)\\
&=& \lim_{\delta\rightarrow 0}\sup_{x'\in\P(x, \delta)}
\mbox{Precision of input $x'$}-\mbox{Precision of output $f(x')$}.
\end{eqnarray*}
Similarly, $\log\C_j(x)$ is called the {\bf Worst-case
Component-wise Loss of Precision}. Besides, we define {\bf
Stochastic Norm-wise Loss of Precision} and {\bf Stochastic
Component-wise Loss of Precision} as follows.
\begin{eqnarray*}
\LNA(x)&=&\lim_{\delta\rightarrow 0}\,\underset {x'\sim\P(x,
\delta)}{\E} \log \frac{\|f(x')-f(x)\|}{\delta\|f(x)\|}\,\, \mbox{ and }\\
\LCA_j(x)&=&\lim_{\delta\rightarrow 0}\,\underset{x'\sim\CP(x,
\delta)}{\E} \log \frac{|f_j(x')-f_j(x)|}{\delta|f_j(x)|}.
\end{eqnarray*}

In short, both condition numbers and Loss of Precision are numbers
telling us how trustable is the computed output when there is
round-off errors. If these numbers are large, the output is not
accurate and we should not trust the output. Otherwise, the computed
output should be accurate and we can trust it.

%\section{What is our goal? What has been done? What is the difference between our results and other researchers'?}
The main goal of this paper is compare the worst-case condition
numbers with the stochastic condition numbers, i.e.
$$
\begin{array}{rclcrcl}
\N(x)&\Vs& \NA(x)&,&\C(x)&\Vs&\CA(x),\\
\log \N(x)& \Vs& \LNA(x)& \mbox{ and }& \log \C(x)& \Vs& \LCA(x).
\end{array}
$$
%Similar results have been done for different Linear Algebra (LA)
%problems. But, to the best of our knowledge, it is the first paper
%providing results which hold for general computation (not only LA)
%problems. There are many different LA problems (solving system of
%equations, finding inverse, determinant, eigenvalues so on and so
%forth). Let alone problems in Linear Programming (See \cite{CCP03}
%and \cite{CC04}. We think it is worthwhile to give results which
%hold for general computation problems.
The theorem \ref{main1} below is one of our main results. It
compares $\NA(x)$ and $\N(x)$. Denote by $e$ the base of the
logarithm of $\ln(\cdot)$.
\begin{theorem}\label{main1}
For any general computational problem with input $x\in\R^m$ and
output $y\in\R^n$, let $k=\min\{m, n\}$ then
$$
\frac{1}{e\sqrt m}\leq\frac{\NA(x)}{\N(x)}\leq \sqrt\frac{k}{{m+2}}
\qquad\mbox{ and}
$$
$$
\frac{-\log m}{2} -\log e\leq\LNA(x)-\log\N(x)\leq \frac{\log
k-\log(m+2)}{2}.
$$
\end{theorem}
Similar results can be found in \cite{WWWS86} and  \cite{S90}. In
this paragraph, we explain the differences between our theorem
\ref{main1} and results in other papers. We write $x\sim N(\mu,
\Sigma)$ if $x$ follows the multivariate normal distribution with
mean $\mu$ and variance-covariance matrix $\Sigma$. In \cite{S90},
the following quality (which is different from our definitions of
condition numbers) was studied.
\begin{eqnarray*}
\underset{(x'-x)\sim N(0, \Sigma)}{\E} \|[\nabla f_1(x), ...,\nabla
f_n(x)]^T (x'-x)\|,
\end{eqnarray*}
where $\Sigma$ can be any variance-covariance  matrix and $\nabla
f_j(x)$ is the gradient of $f_j(x)$. So, their results are
completely different from ours and depend on $\Sigma$. The above
quantity was studied since, by Taylor expansion, when $x'$ is close
to $x$,
\begin{eqnarray*}
 \|[\nabla f_1(x), ...,\nabla
f_n(x)]^T (x'-x)\| \approx \|f(x')-f(x)\|.
\end{eqnarray*}
The results in \cite{WWWS86} hold only for the problem of solving
system of linear equations.  Our theorem \ref{main1} holds for
general computational problems. Besides, the output $f(x)$ in
\cite{WWWS86} was considered to be a real number in $\R$. In this
paper, the output is considered as a real vector in $\R^n$.

Similar to theorem \ref{main1} above,  the corollary \ref{co1} below
compares $\NA(x)$ and $\N(x)$. Comparing with theorem \ref{main1},
corollary \ref{co1} is less explicit and less general (only holds
when $n=1$). But, it provides {\bf equality result instead of
inequality}. Both theorem \ref{main1} and corollary \ref{co1} will
be proved in section \ref{s1}.
\begin{corollary}\label{co1}
For any general computational problem with input $x\in\R^m$ and
output $y\in\R$,
$$
\frac{\NA(x)}{\N(x)}=\left\{
\begin{array}{ll}
\left(\frac{(m)(m-2)...1}{(m+1)(m-1)...2}\right)   \qquad&\mbox{if $m$ is odd}\\
\left(\frac{(m)(m-2)...2}{(m+1)(m-1)...1}\right)\left(\frac{2}{\pi}\right)
\qquad&\mbox{if $m$ is even}
\end{array}
\right.\qquad\mbox{ and}
$$
$$
\frac{\LNA(x)-\log\N(x)}{\log e}=\left\{ \begin{array}{ll}
-\frac{1}{m}-\frac{1}{m-2}-...-\frac{1}{3}-1\qquad&\mbox{ if $m$ is
odd}\\
-\frac{1}{m}-\frac{1}{m-2}-...-\frac{1}{2}-\ln 2\quad&\mbox{ if $m$
is even}
\end{array}
\right..
$$
\end{corollary}
The theorem \ref{main2} below compares $\C(x)$ and $\CA(x)$.
\begin{theorem}\label{main2}
For any general computational problem with input $x\in\R^m$ and
output $y\in\R^n$, if $m>1$,
$$
\frac{-\log (m-1)}{2}-\frac{\log 3}{2}- (1+\varepsilon_m)\log
e<\LCA_j(x) - \log\C_j(x)\leq -1\, \mbox{ and }
$$
$$
\frac{e^{-(1+\varepsilon_m)}}{\sqrt
{3(m-1)}}<\frac{\CA_j(x)}{\C_j(x)}\leq \frac{1}{2} \qquad\mbox{
where }\qquad\varepsilon_m = \frac{2+2\ln m}{\sqrt {m-1}}.
$$
\end{theorem}
\noindent {\bf Note:} It can be easily shown that, $\LCA_j(x) =
\log\C_j(x)-1$ and $\CA_j(x)=0.5\times\C_j(x)$ when $m=1$.

Similar results can also be found in \cite{WWWS86}. Just like our
theorem \ref{main2}, the result in \cite{WWWS86} depends on a
constant $\varepsilon'_m$ which approach to  $0$ as $m\rightarrow
\infty$. But, the speed of convergency was not discussed. Let alone
a formula computing the value (or bound) of $\varepsilon'_m$ for
general $m$. Our theorem \ref{main2} only depends on the size of the
input and output ($m$ and $n$). Once again, in \cite{WWWS86}, only
one problem (solving system of linear equations) was considered. In
this paper, we consider general computational problems. Theorem
\ref{main2} will be proved in
section \ref{s2}.\\
\noindent {\bf Note:} In practice, $\frac{\log m}{2}$ is not very
large ($\frac{\log m}{2}=20$ when the number of data input $m =
1.0995... \times 10^{12}$). So, from the theorems above, we claim
that the value difference between worst case and stochastic loss of
precision are normally very small in practice.

\section{Proof of Corollary \ref{co1} and Theorem \ref{main1}}\label{s1}
For any $c\in\R^m$ and $r\in\R$, let the ball centered at $c$ and
with radius $r$ be
$$
B^m(c,r)=\{u\in\R^m:\, \|u-c\|\leq r\}
$$
and let the sphere be
$$
S^{m-1}(c,r)=\{u\in\R^m:\, \|u-c\|=r\}.
$$
\begin{lemma}\label{l1}
If $u\sim B^m(0,1)$, then
$$
\E(\|u\|)=\frac{m}{m+1},\, \E(\|u\|^2)=\frac{m}{m+2},\mbox{ and }\,
\E(\ln\|u\|)=\frac{-1}{m}.
$$
\end{lemma}
\begin{proof}
Since $u\sim B^m(0,1)$,
\begin{eqnarray*}
\Prob(\|u\|<r) =\frac{\Vol B^m(0,r)}{\Vol B^m(0,1)}=r^m.
\end{eqnarray*}
So, the p.d.f. (probability density function) of $\|u\|$ is
\begin{eqnarray*}
f_{\|u\|}(r) =\frac{d}{dr}\left(r^m\right)=mr^{m-1}.
\end{eqnarray*}
By the definition of expectation and integration by parts,
\begin{eqnarray*}
\E(\|u\|) &=&\int_0^1 mr^mdr\,=\,\frac{m}{m+1}.\\
\E(\|u\|^2) &=&\int_0^1
mr^{m+1}dr\,=\,\frac{m}{m+2}.\\
\E(\ln\|u\|) &=&\int_0^1 mr^{m-1}\ln(r)\,dr \,=\,\frac{-1}{m}.
\end{eqnarray*}
\end{proof}
For any vectors $u, v\in\R^m/\{0\}$, let the angle between $u$ and
$v$ be
$$
\vartheta(u,v)=\arccos\left(\frac{u^T\, v}{\|u\|\,
\|v\|}\right)\in[0, \pi].
$$
Suppose $u$ is fixed and $v/\|v\|\sim S^{m-1}(0,1)$. From
\cite{CCH05}, the p.d.f. (probability density function) of
$\vartheta(u,v)$ is
$$
f_{\vartheta(u,v)}(t)=\frac{(\sin(t))^{m-2}}{I_{m-2}(\pi)}, \mbox{
where } I_m(T) = \int_0^T(\sin(t))^{m}dt.
$$
By integration by part, it can be shown that
\begin{equation}\label{eq2}
I_m\left(\frac{\pi}{2}\right)=\left(\frac{m-1}{m}\right)I_{m-2}\left(\frac{\pi}{2}\right).
\end{equation}
It is easy to check that
\begin{equation}\label{eq3}
I_0\left(\frac{\pi}{2}\right)=\frac{\pi}{2}\qquad\mbox{ and }\qquad
I_1\left(\frac{\pi}{2}\right)=1.
\end{equation}
Combining equations (\ref{eq2}) and (\ref{eq3}) , for $m\geq 2$,
\begin{equation}\label{eq1}
I_m\left(\frac{\pi}{2}\right)= \left\{
\begin{array}{ll}
\left(\frac{(m-1)(m-3)...2}{m(m-2)...1}\right)   \qquad&\mbox{if $m$
is odd}\\
\left(\frac{(m-1)(m-3)...1}{m(m-2)...2}\right)\left(\frac{\pi}{2}\right)
\qquad&\mbox{if $m$ is even}
\end{array}
\right..
\end{equation}
\begin{lemma}\label{l2}
Suppose $u$ is a fixed vector in $\R^m$ and $v/\|v\|\sim
S^{m-1}(0,1)$. Then, for $m\geq 3$,
$\E(|\cos\vartheta(u,v)|^2)=1/m$,
$$
\E(|\cos\vartheta(u,v)|)=\left\{
\begin{array}{ll}
\left(\frac{(m-2)(m-4)...1}{(m-1)(m-3)...2}\right)   \qquad&\mbox{if $m$ is odd}\\
\left(\frac{(m-2)(m-4)...2}{(m-1)(m-3)...1}\right)\left(\frac{2}{\pi}\right)
\qquad&\mbox{if $m$ is even}
\end{array}
\right.\qquad\mbox{ and}
$$
$$
\E(\ln|\cos\vartheta(u,v)|)=\left\{ \begin{array}{ll}
-\frac{1}{m-2}-\frac{1}{m-4}-...-\frac{1}{3}-1\qquad&\mbox{ if $m$
is
odd}\\
-\frac{1}{m-2}-\frac{1}{m-4}-...-\frac{1}{2}-\ln 2\qquad&\mbox{ if
$m$ is even}
\end{array}
\right..
$$
\end{lemma}
\begin{proof}
By the definition of Expectation,
\begin{eqnarray*}
\E(|\cos\vartheta(u,v)|) &=&\int_0^\pi |\cos
(t)|f_{\vartheta(u,v)}(t)dt=\int_0^\pi |\cos (t)|\frac{(\sin(t))^{m-2}}{I_{m-2}(\pi)}dt\\
&=&\int_0^{\pi/2} \cos (t)\frac{(\sin(t))^{m-2}}{I_{m-2}(\pi/2)}dt=\int_0^{\pi/2} \frac{(\sin(t))^{m-2}}{I_{m-2}(\pi/2)}d\sin(t)\\
&=&\frac{1}{(m-1)I_{m-2}(\pi/2)}.
\end{eqnarray*}
By equation (\ref{eq1}), for $m\geq 3$,
\begin{eqnarray*}
\E(|\cos\vartheta(u,v)|) &=&\left\{
\begin{array}{ll}
\left(\frac{(m-2)(m-4)...1}{(m-1)(m-3)...2}\right)   \qquad&\mbox{if $m$ is odd}\\
\left(\frac{(m-2)(m-4)...2}{(m-1)(m-3)...1}\right)\left(\frac{2}{\pi}\right)
\qquad&\mbox{if $m$ is even}
\end{array}
\right..
\end{eqnarray*}
Similarly,
\begin{eqnarray*}
\E((\cos\vartheta(u,v))^2) &=&\int_0^\pi (\cos (t))^2
f_{\vartheta(u,v)}(t)dt=\int_0^\pi (\cos
(t))^2\frac{(\sin(t))^{m-2}}{I_{m-2}(\pi)}dt\\
&=&\int_0^\pi \left(1-(\sin
(t))^2\right)\frac{(\sin(t))^{m-2}}{I_{m-2}(\pi)}dt\\
&=&\int_0^\pi \frac{(\sin(t))^{m-2}-(\sin(t))^m}{I_{m-2}(\pi)}dt= \frac{I_{m-2}(\pi)-I_{m}(\pi)}{I_{m-2}(\pi)}\\
&=&
\frac{I_{m-2}(\pi/2)-I_{m}(\pi/2)}{I_{m-2}(\pi/2)}=1-\frac{m-1}{m}=\frac{1}{m}.
\end{eqnarray*}
The second equality above is due to equation (\ref{eq2}). Besides,
let
$$
J_m(T)=\int_0^T(\sin(t))^{m}\ln|\cos(t)|\,dt.
$$
Then,
\begin{eqnarray*}
J_m\left(\frac{\pi}{2}\right) &=&\int_0^\frac{\pi}{2}
(\sin(t))^{m}\ln\cos (t)dt=-\int_0^\frac{\pi}{2}
(\sin(t))^{m-1}\ln\cos
(t)d(\cos(t))\\
&=&\int_0^{\pi/2} \cos(t)\,d\left((\sin(t))^{m-1}\ln\cos (t)\right)\\
&=&\int_0^{\pi/2} -(\sin(t))^{m}+(m-1)(\cos(t))^2(\sin(t))^{m-2}\ln\cos (t)dt\\
&=&-I_m(\pi/2)+(m-1)J_{m-2}(\pi/2)-(m-1)J_{m}(\pi/2).
\end{eqnarray*}
Combining the above and equation (\ref{eq2}),
\begin{equation}\label{eq4}
\frac{J_m(\pi/2)}{I_m(\pi/2)}=\frac{J_{m-2}(\pi/2)}{I_{m-2}(\pi/2)}-\frac{1}{m}.
\end{equation}
%Besides,
%\begin{eqnarray*}
%J_0\left(\frac{\pi}{2}\right) &=&\int_0^{\pi/2} \ln\cos (t)dt=\int_0^{\pi/2} \ln\sin (t)\,dt.\\
%2J_0\left(\frac{\pi}{2}\right) &=&\int_0^{\pi} \ln\sin
%(t)dt=\int_0^{\pi} \ln \left(2\sin \left(\frac{t}{2}\right)\cos
%\left(\frac{t}{2}\right)\right)dt.
%&=&\pi\ln 2+4J_0\left(\frac{\pi}{2}\right)\\
% &=&\int_0^{\pi} \ln 2+\ln\sin (t/2)+\ln\cos (t/2)dt\\
% &=&\pi\ln 2+\int_0^{\pi} \ln\sin (t/2)dt+\int_0^{\pi} \ln\cos (t/2)dt\\
% &=&\pi\ln 2+2\int_0^{\pi} \ln\sin (t/2)d(t/2)+2\int_0^{\pi} \ln\cos (t/2)d(t/2)\\
% &=&\pi\ln 2+2\int_0^{\pi/2} \ln\sin tdt+2\int_0^{\pi/2} \ln\cos tdt\\
% &=&\pi\ln 2+4J_0(\pi/2)\\
%\end{eqnarray*}
Besides, it is easy to check that
\begin{equation}\label{eq5}
J_0\left(\frac{\pi}{2}\right) =\frac{-\pi\ln 2}{2}\qquad\mbox { and
}\qquad J_1\left(\frac{\pi}{2}\right) = -1.
\end{equation}
Combining equations (\ref{eq3}), (\ref{eq4}) and (\ref{eq5}), we
have
$$
\frac{J_m(\pi/2)}{I_m(\pi/2)}=\left\{ \begin{array}{ll}
-\frac{1}{m}-\frac{1}{m-2}-...-\frac{1}{3}-1\qquad&\mbox{ if $m$ is
odd}\\
-\frac{1}{m}-\frac{1}{m-2}-...-\frac{1}{2}-\ln 2\qquad&\mbox{ if $m$
is even}
\end{array}
\right..
$$
The proof is completed since
$\E(\ln|\cos\vartheta(u,v)|)=\frac{J_{m-2}(\pi/2)}{I_{m-2}(\pi/2)}$.
\end{proof}
\subsection{Proof of corollary \ref{co1}}
\begin{proof}
Denote by $\nabla f(x)$ the gradient of $f$. By Taylor expansion,
\begin{equation}\label{eq6}
f(x')=f(x)+(x'-x)^T\nabla f(x)+O(\|x'-x\|^2).
\end{equation}
Combining the definitions of $\N(x)$ and $\P(x, \delta)$ and
equation (\ref{eq6}),
\begin{eqnarray}
\N(x) &=&
%\lim_{\delta\rightarrow 0}\sup_{x'\in\P(x, \delta)}\frac{|f(x')-f(x)|}{\delta|f(x)|}
\lim_{\delta\rightarrow 0}\sup_{x'\in\P(x,
\delta)}\frac{|(x'-x)^T\nabla f(x)|}{\delta|f(x)|}
%&=&\lim_{\delta\rightarrow 0}\sup_{x'\in\P(x, \delta)}\frac{\delta\, \|x\| \,\|\nabla f(x)\|}{\delta|f(x)|}
=\frac{\|x\|\,\|\nabla f(x)\|}{|f(x)|}.\label{eq10}
\end{eqnarray}
Combining the definitions of $\NA(x)$ and $\P(x, \delta)$  and
equation (\ref{eq6}),
\begin{eqnarray}
&&\NA(x)\,=\,\,
%\lim_{\delta\rightarrow 0}\,\underset{x'\sim\P(x,\delta)}{\E}\left(\frac{|f(x')-f(x)|}{\delta|f(x)|}\right)
\lim_{\delta\rightarrow 0}\underset{x'\sim\P(x, \delta)}{\E}\left(
\frac{(x'-x)^T\nabla f(x)}{\delta|f(x)|}\right)\\
&=&\|\nabla f(x)\|\, \lim_{\delta\rightarrow 0}\underset{x'\sim\P(x,
\delta)}{\E}\left( \frac{\|x'-x\|\times|\cos\vartheta(\nabla f(x),
x'-x)|}{\delta|f(x)|}\right).\label{eq7}
\end{eqnarray}
Since $x'\sim\P(x, \delta)=\B(x,\delta\|x\|)$,
%\begin{eqnarray}
%\NA(x) &=&\|\nabla f(x)\|\times\lim_{\delta\rightarrow
%0}\E_{x'\sim\P(x,
%\delta)}|\cos\vartheta(\nabla f(x), x'-x)|\times\\
%&&\E_{x'\sim\P(x, \delta)}\left(
%\frac{\|(x'-x)\|}{\delta|f(x)|}\right).
%\end{eqnarray}
by lemma \ref{l1},
\begin{eqnarray}
\underset{x'\sim\P(x,
\delta)}{\E}\|x'-x\|=\frac{m\delta\,\|x\|}{m+1}.\label{eq8}
\end{eqnarray}
By lemma \ref{l2},
\begin{eqnarray}
\underset{x'\sim\P(x, \delta)}{\E}|\cos\vartheta(\nabla f(x),
x'-x)|=\left\{
\begin{array}{ll}
\left(\frac{(m-2)(m-4)...1}{(m-1)(m-3)...2}\right)   &\mbox{ if $m$ is odd}\\
\left(\frac{(m-2)(m-4)...2}{(m-1)(m-3)...1}\right)\left(\frac{2}{\pi}\right)
&\mbox{ if $m$ is even}
\end{array}
\right..\label{eq9}
\end{eqnarray}
Combining equations (\ref{eq7}), (\ref{eq8}), (\ref{eq9}) and
(\ref{eq10}), we have
$$
\frac{\NA(x)}{\N(x)}=\left\{
\begin{array}{ll}
\left(\frac{(m)(m-2)...1}{(m+1)(m-1)...2}\right)   \qquad&\mbox{if $m$ is odd}\\
\left(\frac{(m)(m-2)...2}{(m+1)(m-1)...1}\right)\left(\frac{2}{\pi}\right)
\qquad&\mbox{if $m$ is even}
\end{array}
\right..
$$
Similarly, applying lemmas \ref{l1} and \ref{l2}, it can be shown
that
$$
\frac{\LNA(x)-\ln\N(x)}{\ln e}=\left\{ \begin{array}{ll}
-\frac{1}{m}-\frac{1}{m-2}-...-\frac{1}{3}-1\qquad&\mbox{ if $m$ is
odd}\\
-\frac{1}{m}-\frac{1}{m-2}-...-\frac{1}{2}-\ln 2\quad&\mbox{ if $m$
is even}
\end{array}
\right..
$$
\end{proof}
\subsection{Proof of Theorem \ref{main1}}
\begin{proof}
Denote by $f_j(x)$ the $j$th entry of $f(x)$. Denote by $\nabla
f_j(x)$ the gradient of $f_j$. By Taylor Expansion,
$$
f_j(x')=f_j(x)+(x'-x)^T\nabla f_j(x)+O(\|x'-x\|^2).
$$
Let $G=[\nabla f_1(x), \nabla f_2(x), ..., \nabla
f_n(x)]\in\R^{m\times n}$. So,
\begin{equation}\label{eq11}
f(x')=f(x)+G^T(x'-x)+O(\|x'-x\|^2).
\end{equation}
Combining the definitions of $\N(x)$ and $\P(x, \delta)$ and
equation (\ref{eq11}),
\begin{eqnarray}
\N(x) &=& \lim_{\delta\rightarrow 0}\sup_{x'\in\P(x,
\delta)}\frac{\|G^T(x'-x)\|}{\delta\|f(x)\|}
=\frac{\|x\|\,\|G\|}{\|f(x)\|}.\label{eq12}
\end{eqnarray}
Combining the definition of $\NA(x)$ and equation (\ref{eq11}),
\begin{eqnarray}
&&\NA(x)\,=\,\, \lim_{\delta\rightarrow 0}\underset{x'\sim\P(x,
\delta)}{\E}\left(
\frac{\|G^T(x'-x)\|}{\delta\|f(x)\|}\right).\label{eq13}
\end{eqnarray}
Let $UDV$ be the singular value decomposition of $G^T$, i.e. $U,
V\in\R^{n\times n}$ are orthogonal matrices, $D\in\R^{n\times m}$ is
a diagonal matrix with entries $\sigma_1, \sigma_2, ...\sigma_k$ on
its diagonal where  $k=\min\{m,n\}$,
\begin{equation}
G^T=UDV \qquad\mbox{ and }\qquad \|G\|=\sigma_1\geq \sigma_2\geq...
\sigma_k\geq 0.\label{eq14}
\end{equation}
Since $U$ is orthogonal, by equations (\ref{eq13}) and (\ref{eq14}),
$$
\NA(x)= \lim_{\delta\rightarrow 0}\underset{x'\sim\P(x,
\delta)}{\E}\left( \frac{\|UDV(x'-x)\|}{\delta\|f(x)\|}\right) =
\lim_{\delta\rightarrow 0}\underset{x'\sim\P(x, \delta)}{\E}\left(
\frac{\|DV(x'-x)\|}{\delta\|f(x)\|}\right). \label{eq15}
$$
Let $x''=V(x'-x)$. By the definition of $\P(x, \delta)$, $x''\sim
B^m(0, \delta\|x\|)$.  So,
\begin{eqnarray}
\NA(x) & = &\lim_{\delta\rightarrow 0}\underset{x''\sim B^m(0,
\delta\|x\|)}{\E}\left( \frac{\|Dx''\|}{\delta\|f(x)\|}\right).
\label{eq16} \\
\NA(x)^2&\leq& \lim_{\delta\rightarrow 0}\underset{x''\sim B^m(0,
\delta\|x\|)}{\E}\left(
\frac{\sigma_1^2{x''}_1^2+\sigma_2^2{x''}_2^2+...+\sigma_k^2{x''}_k^2}{\delta^2\|f(x)\|^2}\right)\\
&\leq& \sigma_1^2\,\lim_{\delta\rightarrow 0}\underset{x''\sim
B^m(0, \delta\|x\|)}{\E}\left(
\frac{{x''}_1^2+{x''}_2^2+...+{x''}_k^2}{\delta^2\|f(x)\|^2}\right).
\end{eqnarray}
Let $v$ be the vector in $\R^m$ with the first $k$ entries equal to
$1$ and $0$ elsewhere.
\begin{eqnarray}
\NA(x)^2 &\leq& \sigma_1^2\,\lim_{\delta\rightarrow
0}\underset{x''\sim B^m(0, \delta\|x\|)}{\E}\left(
\frac{(v^Tx'')^2}{\delta^2\|f(x)\|^2}\right)\\
&=& \sigma_1^2\|v\|^2\,\lim_{\delta\rightarrow 0}\underset{x''\sim
B^m(0, \delta\|x\|)}{\E}\left(
\frac{(\cos\theta(v,x''))^2\|x''\|^2}{\delta^2\|f(x)\|^2}\right)\\
&=& \left( \frac{\sigma_1^2\|v\|^2\|x\|^2}{(m+2)\|f(x)\|^2}\right)
\qquad\mbox{By lemmas \ref{l1} and  \ref{l2}}.\qquad\label{eq17}
\end{eqnarray}
Since $\sigma_1 =\|G\|$ and $\|v\|=\sqrt k$, by equations
(\ref{eq12}) and (\ref{eq17})
$$
\frac{\NA(x)}{\N(x)}\leq\sqrt \frac{ k}{m+2}.
$$
Since $\log(\cdot)$ is concave function, %$E(\log X)\leq \log E(X)$. Thus,
$$
\LNA(x)-\log \N(x)\leq \frac{\log k-\log (m+2)}{2}.
$$
On the other hand, by equation (\ref{eq16})
\begin{eqnarray*}
\LNA(x)&\geq& \lim_{\delta\rightarrow 0}\underset{x''\sim B^m(0,
\delta\|x\|)}{\E}\log\left(
\frac{\sigma_1{x''}_1}{\delta\|f(x)\|}\right).
\end{eqnarray*}
Let $e_1$ be the vector in $\R^m$ with the first entry to $1$ and
$0$ elsewhere.
\begin{eqnarray*}
\LNA(x)&\geq& \lim_{\delta\rightarrow 0}\underset{x''\sim B^m(0,
\delta\|x\|)}{\E}\log\left(
\frac{\sigma_1e_1^Tx''}{\delta\|f(x)\|}\right)\\
&=& \lim_{\delta\rightarrow 0}\underset{x''\sim B^m(0,
\delta\|x\|)}{\E}\log\left( \frac{\sigma_1\|x''\|\, |\cos\theta(e_1,
x'')|}{\delta\|f(x)\|}\right).
\end{eqnarray*}
By equation (\ref{eq12}) and lemmas \ref{l1} and \ref{l2},
\begin{eqnarray*}
 \frac{\LNA(x)-\log\N(x)}{\log e}&\geq&\left\{ \begin{array}{ll}
-\frac{1}{m}-\frac{1}{m-2}-...-\frac{1}{3}-1\,&\mbox{ if $m$ is
odd}\\
-\frac{1}{m}-\frac{1}{m-2}-...-\frac{1}{2}-\ln 2\,&\mbox{ if $m$ is
even}
\end{array}
\right..
\end{eqnarray*}
Since $\ln m=$ the area of the region $\{(x,y):\,1\leq x\leq n,
0\leq y\leq 1/x\}$,
\begin{eqnarray*}
 \frac{\LNA(x)-\log\N(x)}{\log e}
  &\geq&\left\{ \begin{array}{ll} -\frac{1}{2}\ln
m-1\qquad&\mbox{ if $m$ is
odd}\\
-\frac{1}{2}(\ln m-\ln 2)-\frac{1}{2}-\ln 2\,&\mbox{ if $m$ is even}
\end{array}
\right.\\
&\geq&-\frac{1}{2}\ln m-1.\\
%\LNA(x)-\log\N(x)&\geq&-\frac{1}{2}\log m-\log e.\\
\frac{\NA(x)}{\N(x)}&\geq&\frac{1}{e\sqrt m}\quad \mbox{since
$\log(\cdot)$ is a concave function}.
\end{eqnarray*}
\end{proof}

\section{Proof of Theorem \ref{main2}}\label{s2}
We write $Z\sim N(0,1)$ if $Z$ is random variable following standard
normal  distribution. Below is the well-known Berry-Esseen theorem
(See \cite{F72}).
%theorem from %\cite{K08}.
\begin{theorem}
Let $u_1, ..., u_m$ be i.i.d. random variables with $\E(u_1)=0 $,
$\E(u_1^2)=\sigma $ and $\E(|u_1|^3)=\rho<\infty$. Then, for any
real number $a$,
$$
\left|\Prob\left(a<\frac{u_1+ ...+ u_m}{\sigma\sqrt m}\right) -
\Prob (a<Z)\right|\leq \frac{\ci\rho}{\sigma^3\sqrt m},
$$
where $Z\sim N(0,1)$ and $\ci$ is a universal constant (independent
of $m$).
\end{theorem}
Calculated values of the constant $\ci$ have decreased markedly over
the years, from $7.59$ (Esseen's original bound) to $0.7975$ in 1972
(by P. van Beeck). The best current bound is $0.7655$ (by I. S.
Shiganov in 1986). The lemma below follows Berry-Esseen theorem
immediately.
\begin{lemma}\label{conver}
Let $Z\sim N(0,1)$ and $u_1, ..., u_m\sim[-1, 1]$ be i.i.d. random
variables. Then, for any real number $a$,
$$
\left|\Prob\left(a<\frac{u_1+ ...+ u_m}{\sqrt {m/3}}\right) - \Prob
(a<Z)\right|\leq \frac{1}{\sqrt m}.
$$
\end{lemma}
\begin{proof}
Since $u_m\sim[-1, 1]$, $\rho=1/4$ and $\sigma^2=1/3$. By
Berry-Esseen theorem,
$$
\left|\Prob\left(a<\frac{u_1+ ...+ u_m}{\sqrt {m/3}}\right) - \Prob
(a<Z)\right|\leq \frac{3^{1.5}\ci  }{4\sqrt m}\approx\frac{0.9555
\ci }{\sqrt m}.
$$
\end{proof}
\begin{lemma}\label{l7}
Let $\delta, b$ be positive numbers, s.t. $b>1$. Let $Z\sim N(0,
1)$. Then
$$
\delta\ln \delta+\int_0^{b}\Prob(Z>z)\ln
\left|\frac{z+\delta}{z-\delta}\right|dz> 0.
$$
\end{lemma}
\begin{proof}
Let
\begin{eqnarray*}
F(\delta)&=&\int_0^{b}(z+\delta)\ln|z+\delta|-(z-\delta)\ln|z-\delta|\,d\,\Prob(Z>z).
\end{eqnarray*}
Since $\Prob(Z>0)=1/2$, by integration by parts,
\begin{eqnarray*}
&&F(\delta)\,=\,[(b+\delta)\ln|b+\delta|-(b-\delta)\ln|b-\delta|]\Prob(Z>b)\\
&&-\delta\ln\delta-\int_0^{b}\Prob(Z>z)\,d\,[(z+\delta)\ln|z+\delta|-(z-\delta)\ln|z-\delta|].
\end{eqnarray*}
Since $(b+\delta)\ln|b+\delta| =
2\delta\ln|b+\delta|+(b-\delta)\ln|b+\delta|>(b-\delta)\ln|b-\delta|$,
$F(\delta)$
\begin{eqnarray}
&>&-\delta\ln\delta-\int_0^{b}\Prob(Z>z)\,d[(z+\delta)\ln|z+\delta|-(z-\delta)\ln|z-\delta|]\qquad\\
&=&-\delta\ln\delta-\int_0^{b}\Prob(Z>z)\ln
\left|\frac{z+\delta}{z-\delta}\right|dz.\label{eq38}
\end{eqnarray}
Obviously, $F(0)=0$. Besides,
\begin{eqnarray*}
\frac{dF(\delta)}{d\delta}&=&\int_0^{b}\frac{d(z+\delta)\ln|z+\delta|}{d\delta}-\frac{d(z-\delta)\ln|z-\delta|}{d\delta}\,d\,\Prob(Z>z)\\
&=&\int_0^{b}\ln|z+\delta|-\ln|z-\delta|\,d\,\Prob(Z>z)\\
&=&-\int_0^{b}\ln\left|\frac{z+\delta}{z-\delta}\right|f_Z(z)\,dz<0\qquad\mbox{
since $\left|\frac{z+\delta}{z-\delta}\right|\geq 1, \forall z\geq
0$}.
\end{eqnarray*}
So, $F(\delta)\leq 0$. Together with equation (\ref{eq38}), the
proof is completed.
\end{proof}
\begin{lemma}\label{l4}
Let $\delta\leq \sqrt {3m}$ be a positive number. Let $u_1, ...,
u_m\sim[-1, 1]$ be i.i.d. random variables. Then
\begin{eqnarray*}
&&\E\left(\left(\frac{u_1+ ...+ u_m}{\sqrt
{m/3}}+\delta\right)\ln\left|\frac{u_1+ ...+ u_m}{\sqrt
{m/3}}+\delta\right|\right)\\
&>&\frac{-2\delta}{{\sqrt m}}\left(\ln \left(1+\frac{\sqrt
{3m}}{\delta}\right)+1\right). \end{eqnarray*}
\end{lemma}
\begin{proof}
Let $W =(u_1+...+u_{m})\sqrt {3/m}$ and $f_W(w)$ be the probability
density function of $W$. By definition of expectation,
\begin{eqnarray*}
&&\E[(W+\delta)\ln |W+\delta|]= \int_{-\sqrt {3m}}^{\sqrt {3m}}(w+\delta)\ln |w+\delta|f_W(w)dw\\
&=& \int_{-\sqrt {3m}}^0(w+\delta)\ln
|w+\delta|f_W(w)dw+\int_0^{\sqrt {3m}}(w+\delta)\ln
|w+\delta|f_W(w)dw\\
&=&\int_0^{\sqrt {3m}}(w+\delta)\ln |w+\delta|f_W(w)dw-\int^0_{\sqrt
{3m}}(\delta-w)\ln
|\delta-w|f_W(-w)dw\\
&=& \int_0^{\sqrt {3m}}\left[(w+\delta)\ln |w+\delta|-(w-\delta)\ln
|w-\delta|\right]f_W(w)dw\\%\qquad\mbox{since $f_W(-w) =f_W(w)$}\\
&=& -\int_0^{\sqrt {3m}}(w+\delta)\ln |w+\delta|-(w-\delta)\ln
|w-\delta|\,\,d\,\Prob(W>w).
\end{eqnarray*}
Note that: when $w = 0$, $P(W>w) = 0.5$ and $(w+\delta)\ln
|w+\delta|=-(w-\delta)\ln |w-\delta|=\delta\ln \delta$. So, by
integration by parts, $\E[(W+\delta)\ln |W+\delta|]-\delta\ln
\delta$
\begin{eqnarray*}
&=& \int_0^{\sqrt {3m}}\Prob(W>w)\,\,d\, [(w+\delta)\ln
|w+\delta|-(w-\delta)\ln |w-\delta|]\\
&=& \int_0^{\sqrt {3m}}\Prob(W>w)\ln
\left|\frac{w+\delta}{w-\delta}\right|dw.
\end{eqnarray*}
Since $|w+\delta|\geq|w-\delta|$ for all $w>0$, by lemma
\ref{conver}, $\E[(W+\delta)\ln |W+\delta|]$
\begin{eqnarray*}
& \geq& \delta\ln \delta+\int_0^{\sqrt {3m}}\Prob(Z>z)\ln
\left|\frac{z+\delta}{z-\delta}\right|dz-\frac{1}{\sqrt
m}\int_0^{\sqrt {3m}}\ln \left|\frac{w+\delta}{w-\delta}\right|dw\\
& >& -\frac{1}{\sqrt m}\int_0^{\sqrt {3m}}\ln
\left|\frac{w+\delta}{w-\delta}\right|dw\qquad \mbox{by lemma
\ref{l7}}.
\end{eqnarray*}
So, $-\sqrt m\,\, \E[(W+\delta)\ln |W+\delta|]$
\begin{eqnarray*}
&< &\int_0^{\sqrt {3m}}\ln
\left|\frac{w+\delta}{w-\delta}\right|dw\,  =\, \int_0^{\sqrt
{3m}}\ln (w+\delta)dw-\int_0^{\sqrt {3m}}\ln
 \left|w-\delta\right|dw\\
 &=&[w\ln w-w]_\delta^{\sqrt {3m}+\delta}-\int_0^{\delta}\ln
 \left(\delta-w\right)dw-\int_\delta^{\sqrt {3m}}\ln
 \left(w-\delta\right)dw\\
 &=&[w\ln w-w]_\delta^{\sqrt {3m}+\delta}+[w\ln w-w]_\delta^0-[w\ln w-w]_0^{\sqrt {3m}-\delta}\\
 &=&(\sqrt {3m}+\delta)\ln (\sqrt {3m}+\delta)-(\sqrt {3m}-\delta)\ln (\sqrt
 {3m}-\delta)-2\delta\ln\delta\\
 &=&2\delta\ln (\sqrt {3m}+\delta)+(\sqrt {3m}-\delta)\ln \left(\frac{\sqrt
 {3m}+\delta}{\sqrt
 {3m}-\delta}\right)-2\delta\ln\delta.
\end{eqnarray*}
Since $\ln x\leq x-1$ for all $x>0$, $-\sqrt m\,\, \E[(W+\delta)\ln
|W+\delta|]$
\begin{eqnarray*}
&<&2\delta\ln (\sqrt {3m}+\delta)+(\sqrt
{3m}-\delta)\left(\frac{\sqrt
 {3m}+\delta}{\sqrt
 {3m}-\delta}-1\right)-2\delta\ln\delta\\
 &=&2\delta\ln (\sqrt {3m}+\delta)+2\delta-2\delta\ln\delta
 =2\delta\left(\ln \left(1+\frac{\sqrt {3m}}{\delta}\right)+1\right).
 \label{eq26}
\end{eqnarray*}
\end{proof}
\begin{lemma}\label{l5}
If $u_1, ...u_{m+1}\sim [-1,1]$ are i.i.d.,
$$
\E(\ln|u_1+...+u_{m+1}|)=\E((u_1+...+u_m+1)\ln|u_1+...+u_m+1|)-1.
$$\end{lemma}
\begin{proof}
For any fixed $a\in\R$,
\begin{eqnarray}
&&\int_{-1}^1\ln|a+u|du=\int_{a-1}^{a+1}\ln |v|dv\qquad\mbox{(let $v=a+u$)}\qquad\\
&=&\left[v\ln|v|-v\right]_{a-1}^{a+1}=(a+1)\ln|a+1|-(a-1)\ln|a-1|-2.\label{eq33}
\end{eqnarray}
So, by the definition of expectation, $\E(\ln|u_1+...+u_{m+1}|)$
\begin{eqnarray*}
&=&\frac{1}{2^{m+1}}\int_{-1}^1\cdots\int_{-1}^1\ln|u_1+...+u_{m+1}|du_{m+1}\cdots du_1\\
&=&\frac{1}{2^{m+1}}\int_{-1}^1\cdots\int_{-1}^1(u_1+...+u_{m}+1)\ln|u_1+...+u_{m}+1|du_m\cdots du_1\\
&-&\frac{1}{2^{m+1}}\int_{-1}^1\cdots\int_{-1}^1(u_1+...+u_{m}-1)\ln|u_1+...+u_{m}-1|du_m\cdots
du_1\\
&-&\frac{1}{2^{m+1}}\int_{-1}^1\cdots\int_{-1}^12du_m\cdots
du_1\qquad\mbox{by equation (\ref{eq33})}.
\end{eqnarray*}
So, $\E(\ln|u_1+...+u_{m+1}|)+1$
\begin{eqnarray}
&=&\frac{1}{2}\E((u_1+...+u_{m}+1)\ln|u_1+...+u_{m}+1|)\label{eq23}\\
&-&\frac{1}{2}\E((u_1+...+u_{m}-1)\ln|u_1+...+u_{m}-1|).
\end{eqnarray}
Since $u$ and $-u$ follow the same distribution,
\begin{eqnarray}
&&{\E}[(u_1+...+u_{m+1}-1)\ln|u_1+...+u_{m+1}-1|]\\
&=&{\E}[(-u_1-...-u_{m+1}-1)\ln|-u_1-...-u_{m+1}-1|]\\
&=&-{\E}[(u_1+...+u_{m+1}+1)\ln|u_1+...+u_{m+1}+1|].\label{eq24}
\end{eqnarray}
Combining equations (\ref{eq23}) and (\ref{eq24}), the proof is
completed.
\end{proof}
In this section, we follow the definition of $\varepsilon_m$ given
in theorem \ref{main2},
$$
\varepsilon_m =\frac{2+2\ln m}{\sqrt {m-1}}.
$$
\begin{corollary}\label{co2}
If $u_1, ..., u_{m+1}\sim[-1,1]$ are i.i.d. then
\begin{eqnarray*}
\E(\ln|u_1+...+u_{m+1}|)&>& \frac{\ln m}{2}-\frac{\ln 3}{2}-
1-\varepsilon_{m+1}.
\end{eqnarray*}
\end{corollary}
\begin{proof}
Let $W=(u_1+...+u_m)\sqrt {3/m}$.
\begin{eqnarray*}
&&\E((u_1+...+u_m+1)\ln|u_1+...+u_m+1|)\\
&=&\E\left(\left(W\sqrt \frac{m}{3}+1\right)\ln\left|W\sqrt \frac{m}{3}+1\right|\right)\\
&=&\sqrt \frac{m}{3}\, \E\left(\left(W+\sqrt\frac{3}{m}\right)\ln\left(\left|W+\sqrt\frac{3}{m}\,\right|\sqrt \frac{m}{3}\right)\right)\\
&=&\sqrt \frac{m}{3}\,
\E\left(\left(W+\sqrt\frac{3}{m}\right)\ln\left|W+\sqrt\frac{3}{m}\right|\right)
+\frac{\ln (m/3)}{2}\sqrt \frac{m}{3}\, \E\left(W+\sqrt\frac{3}{m}\right)\\
&>&\frac{\ln (m/3)}{2}\sqrt \frac{m}{3}\,
\E\left(W+\sqrt\frac{3}{m}\right)- \frac{2}{{\sqrt m}}\left(\ln
\left(m+1\right)+1\right) \,\mbox{ by lemma \ref{l4}}\\
&=&\frac{\ln (m/3)}{2}- \frac{2}{{\sqrt m}}\left(\ln
\left(m+1\right)+1\right) \qquad\mbox{ Since $\E(W)=0$}.
\end{eqnarray*}
Applying lemma \ref{l5}, the proof is completed.
\end{proof}
Denote by $e_n$ the vector in $\R^n$ with all entries $1$. For any
$a\in\R^m$, denote by $\|a\|_1=|a_1|+...+|a_m|$ the $1-$norm of $a$.
Below is a lemma from \cite{WWWS86},
\begin{lemma}\label{l6}
For any increasing function $\varphi:\, \R\rightarrow\R$, $a\in\R^m$
and $b\in\R$, s.t. $\|a\|_1=m$,
$$
\Prob(|a^Tu|>b)\geq
\Prob(|e_m^Tu|>b)\qquad\mbox{and}\qquad\E(\varphi(|a^Tu|)\geq\E(\varphi(|e_m^Tu|).
$$
\end{lemma}
\subsection{Proof of Theorem \ref{main2}}
\begin{proof}
Denote by $f_j(x)$ the $j$th entry of $f(x)$. Denote by $\nabla
f_j(x)$ the gradient of $f_j$. By Taylor Expansion,
\begin{equation}
f_j(x')=f_j(x)+(x'-x)^T\nabla f_j(x)+O(\|x'-x\|^2).\label{eq34}
\end{equation}
Combining the definition of $\C(x)$ and equation (\ref{eq34}),
\begin{eqnarray*}
\C_j(x) &=& \lim_{\delta\rightarrow 0}\sup_{x'\in\CP(x,
\delta)}\frac{|(x'-x)^T\nabla f_j(x)|}{\delta|f_j(x)|} .
\end{eqnarray*}
For $i=1, ..., m$, let
$$
%u_i=\frac{x_i'-x_i}{\delta|x_i|}\qquad\mbox{and}\qquad
g_i =x_i\times\mbox{the $i$th component of } \nabla f_j(x).
$$
Then, by the definition of $\CP(x, \delta)$, %$u_i\sim[-1,1]$.
%Besides,
\begin{eqnarray}
\C_j(x) &=& \sup_{u\in[-1,1]^m}\frac{|u^Tg|}{|f_j(x)|}=
\frac{\|g\|_1}{|f_j(x)|}.\label{eq37}
\end{eqnarray}
Combining the definition of $\CA(x)$, $\CP(x, \delta)$ and $g$ and
equation (\ref{eq34}),
\begin{eqnarray}
\CA_j(x)&=& \lim_{\delta\rightarrow 0}\underset{x'\sim\CP(x,
\delta)}{\E}\left| \frac{(x'-x)^T\nabla f_j(x)}{\delta
\,f_j(x)}\right|= \underset{u\sim[-1,1]^m}{\E}\left|
\frac{u^Tg}{f_j(x)}\right|.\qquad\\
\frac{\CA_j(x)}{\C_j(x)}&=& \underset{u\in[-1,1]^m}{\E}\left(
\frac{|u^Tg|}{\|g\|_1}\right)\qquad\mbox{by equation
(\ref{eq37}}).\label{eq35}
\end{eqnarray}
Obviously,
\begin{eqnarray}
\underset{u\in[-1,1]^m}{\E}|u^Tg|
&=&\underset{u\in[-1,1]^m}{\E}|u_1g_1+...+u_mg_m|\\
&\leq&\underset{u\in[-1,1]^m}{\E}\left(|u_1g_1|+...+|u_mg_m|\right)\\
&=&|g_1|\underset{u\in[-1,1]^m}{\E}|u_1|+...+|g_m|\underset{u\in[-1,1]^m}{\E}|u_m|\\
&=&0.5\left(|g_1|+...+|g_m|\right) = 0.5\|g\|_1.\label{eq36}
\end{eqnarray}
Combining equations (\ref{eq35}) and (\ref{eq36}),
$\frac{\CA_j(x)}{\C_j(x)}\leq \frac{1}{2}.$ Since $\log(\cdot)$ is a
concave function, $\LCA_j(x)-\log C_j(x)\leq -1$. Combining the
definition of $\LCA_j(x)$, $\CP(x, \delta)$ and $g$ and equation
(\ref{eq34}),
\begin{eqnarray*}
&&\frac{\LCA_j(x)}{\log e} \, =\,\,  \lim_{\delta\rightarrow
0}\underset{x'\sim\CP(x, \delta)}{\E}\ln\left| \frac{(x'-x)^T\nabla
f_j(x)}{\delta
\,f_j(x)}\right|\\
&=&
\ln\frac{\|g\|_1}{m|f_j(x)|}+\underset{u\sim[-1,1]^m}{\E}\ln\left|
u^T\frac{mg}{\|g\|_1}\right|\\
&\geq&
\ln\frac{\|g\|_1}{m|f_j(x)|}+\underset{u\sim[-1,1]^m}{\E}\ln\left|
u^Te\right|\qquad\mbox{ by lemma \ref{l6}}\\
&>& \ln\frac{\|g\|_1}{|f_j(x)|}-\frac{\ln (m-1)}{2}-\frac{\ln
3}{2}-1- \varepsilon_{m}\qquad\mbox{ by corollary
\ref{co2}}\\
&=& \ln\C_j(x)-\frac{\ln (m-1)}{2}-\frac{\ln 3}{2}-1-
\varepsilon_{m}\qquad\mbox{by equation (\ref{eq37}).}
\end{eqnarray*}
That is, $ \LCA_j(x) > \log\C_j(x)-\frac{\log (m-1)}{2}-\frac{\log
3}{2}- (1+\varepsilon_{m})\log e.$ Since $\log(\cdot)$ is a concave
function, $\CA_j(x) >  \frac{\C_j(x)}{\sqrt{
3(m-1)}}e^{-(1+\varepsilon_{m})}$.
\end{proof}

\end{document}